\documentclass[a4paper, 11pt,reqno]{amsart}
\usepackage{amsthm,amssymb,amsmath,graphicx,psfrag,enumitem,mathrsfs}
\usepackage{mathtools}
\usepackage[british]{babel}
\usepackage[babel]{microtype}
\usepackage[margin=1in]{geometry}
\usepackage[textsize=footnotesize]{todonotes}

\usetikzlibrary{backgrounds}

\usepackage[numbers]{natbib}

\makeatletter
\def\NAT@spacechar{~}
\makeatother
    
\usepackage{hyperref}
\usepackage[capitalise]{cleveref}
\hypersetup{
   colorlinks=true,
   citecolor=blue,
   filecolor=blue,
   linkcolor=blue,
   urlcolor=blue
}

\crefname{claim}{Claim}{Claims}
\Crefname{claim}{Claim}{Claims}
\crefname{figure}{Figure}{Figures}
\Crefname{figure}{Figure}{Figures}
\crefname{equation}{equation}{equations}
\Crefname{equation}{Equation}{Equations}

\allowdisplaybreaks

\newtheorem{definition}{Definition}[section]
\newtheorem{claim}{Claim}

\newtheorem{proposition}[definition]{Proposition}
\newtheorem{theorem}[definition]{Theorem}

\newtheorem{lemma}[definition]{Lemma}

\newtheorem{conjecture}[definition]{Conjecture}

\newtheorem*{claim*}{Claim}

\numberwithin{equation}{section}

\newcommand{\comment}[1]{}

\newcommand{\eps}{\varepsilon}
\renewcommand{\epsilon}{\varepsilon}
\newcommand*\diff{\mathop{}\!\mathrm{d}}

\newcommand{\COMMENT}[1]{}

\title[Resilient degree sequences in random graphs]{Resilient degree sequences with respect to Hamilton cycles and matchings in random graphs}

\author[P.~Condon]{Padraig Condon}
\author[A.~Espuny D\'iaz]{Alberto Espuny D\'iaz}
\author[J.~Kim]{Jaehoon Kim}
\author[D.~K\"uhn]{Daniela K\"uhn}
\author[D.~Osthus]{Deryk Osthus}

\thanks{The research leading to these results was partially supported by the EPSRC, grant no. EP/N019504/1 (D.~K\"uhn), 
and by the Royal Society and the Wolfson Foundation (D.~K\"uhn).
The research was  also partially supported by the European Research Council under the European Union's Seventh Framework Programme (FP/2007--2013) / ERC Grant 306349 (J.~Kim and D.~Osthus). }

\date{\today}

\begin{document}

\begin{abstract}
P\'osa's theorem states that any graph $G$ whose degree sequence $d_1 \le \ldots \le d_n$ satisfies $d_i \ge i+1$ for all $i < n/2$ has a Hamilton cycle.
This degree condition is best possible.
We show that a similar result holds for suitable subgraphs $G$ of random graphs, i.e.~we prove a `resilience version' of P\'osa's theorem:
if $pn \ge C \log n$ and the $i$-th vertex degree (ordered increasingly) of $G \subseteq G_{n,p}$ is at least $(i+o(n))p$ for all $i<n/2$, then $G$ has a Hamilton cycle.
This is essentially best possible and strengthens a resilience version of Dirac's theorem
obtained by Lee and Sudakov.

Chv\'atal's theorem generalises P\'osa's theorem and characterises all degree sequences which ensure the existence of a Hamilton cycle.
We show that a natural guess for a resilience version of Chv\'atal's theorem fails to be true.
We formulate a conjecture which would repair this guess, and show that the corresponding degree conditions ensure the existence of a perfect matching in any subgraph of $G_{n,p}$ which satisfies these conditions.
This provides an asymptotic characterisation of all degree sequences which resiliently guarantee the existence of a perfect matching.
\end{abstract}

\maketitle
\thispagestyle{empty}

\section{Introduction}

One of the most well-known and well-studied properties in graph theory is \emph{Hamiltonicity}.
We say that a graph $G$ is \emph{Hamiltonian} whenever it contains a cycle which covers all of the vertices of $G$.
We refer to such a cycle as a \emph{Hamilton cycle}.
The problem of determining whether or not a graph is Hamiltonian is NP-complete~\cite{Karp}.
Thus, the study of Hamiltonicity focuses on finding sufficient conditions, particularly in the form of degree conditions.

In 1952, \citet{Dirac52} proved that every $n$-vertex graph $G$ with minimum degree at least $n/2$ is Hamiltonian.
\citet{Posa62} strengthened this result.
More specifically, a graph $G$ with degree sequence $d_1 \le \ldots \le d_n$ such that $d_i \ge i+1$ for all $i < n/2$ is Hamiltonian.
This is best possible in the sense that the condition $d_i \ge i+1$ cannot be reduced for any $i$.
\citet{Chva72} generalised this further by essentially characterising all degree sequences which guarantee Hamiltonicity:
a graph with degree sequence $d_1 \leq \ldots \leq d_n$ is Hamiltonian if for all $i<n/2$ we have $d_i\geq i+1$ or $d_{n-i}\geq n-i$.

The search for Hamilton cycles in random graphs has also been at the core of the subject (as well as the closely related problem of finding perfect matchings).
\citet{ER1, ER2} showed that the random graph $G_{n,p}$ with $p\geq C\log{n}/n$ a.a.s.~contains a perfect matching (if $n$ is even and $C$ is large enough).
\citet{Posa76} and \citet{Kor76} independently showed that for the same threshold $G_{n,p}$ is a.a.s.~Hamiltonian, and \citet{KS83} determined the exact threshold for $p$.
Remarkably, one can strengthen these results to obtain the following hitting time results.
Consider the following random graph process: given a vertex set of size $n$, add each of the $\binom{n}{2}$ possible edges, one by one, chosen uniformly at random among all edges that have not been added yet.
Then, \citet{BT} showed that a.a.s.~a perfect matching appears as soon as every vertex has degree at least $1$, and \citet{AKS85} and \citet{Bol84} independently proved that a.a.s.~a Hamilton cycle appears as soon as this graph has minimum degree $2$.

One more recent approach to extend the classical extremal results to random graphs is based on the following concept of \emph{resilience}.
The \emph{local resilience} of a graph $G$ with respect to some property $\mathcal{P}$ is the maximum number $r$ such that for any subgraph $H\subseteq G$ with $\Delta(H)< r$, the graph $G\setminus H$ satisfies $\mathcal{P}$.
One may view this concept as a measure of the damage an adversary can commit at each vertex of $G$, without destroying the property $\mathcal{P}$. 
The systematic study of local resilience was initiated by \citet{SV08}.
Restated in this terminology, Dirac's theorem says that the local resilience of the complete graph $K_n$ with respect to Hamiltonicity is $\lfloor n/2 \rfloor$.

This concept of resilience naturally suggests a generalisation of Dirac's theorem in the setting of random graphs.
\citet{LS12} proved that, when $p=C\log{n}/n$ and $C$ is sufficiently large, the local resilience of the random graph $G_{n,p}$ with respect to Hamiltonicity is a.a.s.~at least $(1/2-\epsilon)np$, extending Dirac's theorem to random graphs.
This improved on earlier bounds~\cite{BKS,FK08,SV08}.
Very recently, \citet{Mont17} as well as \citet{NST17} independently obtained a hitting time version of this result (Nenadov, Steger and Truji{\'c} also obtained such a hitting time version for perfect matchings~\cite{NST17}).

Resilience of random graphs with respect to other properties has also been extensively studied.
In particular, the containment of cycles of all possible lengths~\cite{KLS10}, $k$-th powers of cycles of all possible lengths~\cite{NST2}, bounded degree trees~\cite{BCS11}, triangle factors~\cite{BLS}, and bounded degree graphs~\cite{ABET,HLS12} have been considered.
Local resilience with respect to Hamiltonicity has also been studied in other random graph models, such as binomial random directed graphs~\cite{FNNPS17,HSS16,Mont19} and random regular graphs~\cite{BKS11b,CEGKO}.

\citet{LS12} asked for a characterisation of the degree sequences for which the random graph $G_{n,p}$ is resilient with respect to Hamiltonicity, for $p$ close to $\log{n}/n$.
In this paper, we partially answer this question by extending P\'osa's theorem to the setting of random graphs.
We also prove that the obvious extension to a Chv\'atal-type degree condition is false, while some modifications to those conditions suffice to force at least the containment of a perfect matching.
We conjecture that such a modification is also sufficient for Hamiltonicity.

To state our results precisely, we start with the following definition, which generalises the class of graphs whose degree sequences satisfy P\'osa's condition to the setting of random graphs.

\begin{definition}[P\'osa-resilience]\label{posa}
Let $G=G_{n,p}$ and $\eps>0$.
Let $\mathcal{H}^{\eps}_{n,p}$ be the collection of all $n$-vertex graphs $H$ which satisfy the following property:
there is an ordering $v_1,\ldots,v_n$ of the vertices with $d_H(v_1)\geq\ldots\geq d_H(v_n)$ such that, for all $i < n/2$, 
\begin{equation}\label{equa:PosadegreePosa}
d_H(v_i) \leq (n-i)p - \eps np.
\end{equation}
We denote $\mathcal{H}^{\varepsilon}_{n,p}(G)\coloneqq\{H\in\mathcal{H}^{\eps}_{n,p}:H\subseteq G\}$.
We say that $G$ is $\varepsilon$-\emph{P\'osa-resilient} with respect to a property $\mathcal{P}$ if $G\setminus H \in \mathcal{P}$ for all $H \in \mathcal{H}^{\eps}_{n,p}(G)$.
\end{definition}

We can now state our first main result.

\begin{theorem}\label{teor:posa}
For every $\varepsilon>0$, there exists $C>0$ such that, for $p \ge C\log n/n,$ a.a.s.~the random graph $G_{n,p}$ is $\eps$-P\'osa-resilient with respect to Hamiltonicity.
\end{theorem}

Next, we consider the following definition, which generalises the class of graphs whose degree sequences satisfy Chv\'atal's condition to the setting of random graphs.

\begin{definition}[Chv\'atal-resilience]\label{chvatalbad}
Let $G=G_{n,p}$ and $\eps>0$.
Let $\mathcal{H}^{\eps,0}_{n,p}$ be the collection of all $n$-vertex graphs $H$ which satisfy the following property:
there is an ordering $v_1,\ldots,v_n$ of the vertices with  $d_H(v_1)\geq\ldots\geq d_H(v_n)$ such that, for all $i < n/2$, either 
\[
d_H(v_i) \leq (n-i)p - \eps np \qquad
\text{or} \qquad
d_H(v_{n-i}) \leq ip -\eps np.
\]
We denote $\mathcal{H}^{\varepsilon,0}_{n,p}(G)\coloneqq\{H\in\mathcal{H}^{\eps,0}_{n,p}:H\subseteq G\}$.
We say that $G$ is $\eps$-\emph{Chv\'atal-resilient} with respect to a property $\mathcal{P}$ if $G\setminus H \in \mathcal{P}$ for all $H \in \mathcal{H}^{\eps,0}_{n,p}(G)$.
\end{definition}

Surprisingly, unlike the case of P\'osa-resilience, random graphs are not Chv\'atal-resilient with respect to even the containment of perfect matchings.
(We actually prove a stronger result, see \cref{teor:Counter}.)

\begin{theorem}\label{teor:CounterIntro}
For every $0< \eps < 10^{-6}$ there exists $C>0$ such that, for $C\log n/n \leq p\leq 1/25$, a.a.s.~the random graph $G_{n,p}$ is not $\eps$-Chv\'atal-resilient with respect to containing a perfect matching.
\end{theorem}

This leads to the following modified version of \cref{chvatalbad}.
A related concept (i.e.~a shift in the Chv\'atal condition) was considered by \citet{KOT10} in the setting of directed Hamilton cycles.

\begin{definition}[Shifted Chv\'atal-resilience]\label{chvatal}
Let $G=G_{n,p}$ and let $\eps\text{,}\, \delta >0$.
Let $\mathcal{H}^{\eps,\delta}_{n,p}$ be the collection of all $n$-vertex graphs $H$ which satisfy the following property:
there is an ordering $v_1,\ldots,v_n$ of the vertices with $d_H(v_1)\geq\ldots\geq d_H(v_n)$ such that, for all $i < n/2$, either 
\begin{equation}\label{equa:Posadegree}
d_H(v_i) \leq (n-i)p - \eps np
\end{equation}
or
\begin{equation}\label{equa:Chvataldegree}
d_H(v_{n-i-\delta n}) \leq ip -\eps np.
\end{equation}
We denote $\mathcal{H}^{\varepsilon,\delta}_{n,p}(G)\coloneqq\{H\in\mathcal{H}^{\eps,\delta}_{n,p}:H\subseteq G\}$.
We say that $G$ is $(\varepsilon,\delta)$-\emph{Chv\'atal-resilient} with respect to a property $\mathcal{P}$ if $G\setminus H \in \mathcal{P}$ for all $H \in \mathcal{H}^{\eps,\delta}_{n,p}(G)$.
\end{definition}

Note that \eqref{equa:Chvataldegree} is never satisfied for $i< \eps n$. 
The conditions \eqref{equa:Posadegree} and \eqref{equa:Chvataldegree} together imply that
\begin{equation}\label{equa:maxdegree}
d_H(v) \le (1-\eps)np
\end{equation}
for all $H \in \mathcal{H}^{\eps, \delta}_{n,p}$ and all vertices $v$ of $H$.
As $\mathcal{H}_{n,p}^\eps\subseteq\mathcal{H}^{\varepsilon,\delta}_{n,p}$, the same bound holds when considering $\eps$-P\'osa-resilience.

With this new definition of shifted Chv\'atal-resilience we can obtain the following version of Chv\'atal's theorem for random graphs with respect to the containment of perfect matchings.

\begin{theorem}\label{teor:ChvMatch}
For every $\varepsilon>0$, there exists $C > 0$ such that, for $p \ge C\log n/n$, a.a.s.~the random graph $G_{n,p}$ is $(\eps,\eps)$-Chv\'atal-resilient with respect to containing a perfect matching if $n$ is even.
\end{theorem}

We conjecture that \cref{teor:ChvMatch} also holds if perfect matchings are replaced by Hamilton cycles.

\begin{conjecture}
For every $\varepsilon>0$, there exists $C > 0$ such that, for $p \ge C\log n/n$, a.a.s.~the random graph $G_{n,p}$ is $(\eps,\eps)$-Chv\'atal-resilient with respect to Hamiltonicity.
\end{conjecture}

The following simple construction shows that this statement, if true, is essentially best possible.
Let $G = G_{n,p}$ with $p\geq C\log n/n$ for some sufficiently large $C$.
Given any $\eps n\leq i < n/2$, fix disjoint sets $X,Y\subseteq V$ of sizes $i$ and $n-i$, respectively, and let $H$ be the induced bipartite subgraph between $X$ and $Y$.
One can then prove that a.a.s.
\[
d_H(x) \leq (n-i)p + \eps np \qquad \text{ and } \qquad d_H(y) \leq ip + \eps np
\]
for all $x\in X$ and $y\in Y$.
Thus, $H$ is `close' to satisfying the conditions of \cref{chvatal}, and it is clear that $G\setminus H$ is not Hamiltonian since it is disconnected.
The same construction shows that \cref{teor:posa} is essentially best possible (in the sense that we cannot significantly relax the degree condition) and that \cref{teor:ChvMatch} is essentially best possible when considering odd $i$.

Investigating resilience with respect to degree sequences is natural not only for perfect matchings and Hamilton cycles, but also for other properties. 
Several results on degree sequences forcing given substructures have been obtained in the classical setting (see e.g.~\cite{ST17,Tregs16} for such results involving P\'osa-type degree sequences and \cite{KT13} for Chv\'atal-type degree sequences).
It would be interesting to see if one can obtain resilience versions (for random graphs) of some of these results.


\section{Preliminaries}

\subsection{Notation}

For $n \in \mathbb{N},$ we denote $[n] \coloneqq \{1, \ldots , n\}$.
The constants which appear in hierarchies are chosen from right to left.
That is, whenever we use a hierarchy $0 < 1/n \ll a \ll b \le 1,$ we mean that there exist non-decreasing functions $f\colon [0, 1) \to [0,1)$ and $g\colon [0, 1) \to [0,1)$ such that the result holds for all $0 \le a$, $b \le 1$ and all $n \in \mathbb{N}$ with $a \le f(b)$ and $1/n \le g(a)$.
We will not calculate these functions explicitly.

We use \emph{a.a.s.}~as an abbreviation for \emph{asymptotically almost surely}. 
Whenever we claim that a result holds a.a.s.~for $G_{n,p}$, we mean that the probability that our result holds tends to one as $n$ tends to infinity. 
For the purpose of clarity, we will ignore rounding issues when dealing with asymptotic statements, whenever the values we consider tend to infinity with $n$.

Given an $n$-vertex graph $G$ we define $e(G)\coloneqq |E(G)|$.
Given a set $A \subseteq V(G)$ we denote by $e_G(A)$ the number of edges in $G$ whose endpoints are both in $A$.
Given another set $B \subseteq V(G)$ we denote by $E_G(A,B)$ the set of edges of $G$ with one endpoint in $A$ and the other in $B$ (note that $A$ and $B$ are allowed to have a nonempty intersection), and $e_G(A,B)\coloneqq|E_G(A,B)|$.
Given any $v\in V(G)$, we will write $e_G(v,A)\coloneqq e_G(\{v\},A)$.
Sometimes it will be useful to consider $e_G'(A,B)\coloneqq e_G(A,B)+e_G(A\cap B)$.
We will often refer to the graph $G[V(G) \setminus A]$, which we denote as $G - A$.
If $A$ and $B$ are disjoint, the notation $G[A,B]$ will refer to the induced bipartite subgraph with vertex classes $A$ and $B$.
We denote the \emph{neighbourhood} of $A$ as $N_G(A)\coloneqq\{v\in V(G):e_G(\{v\},A)>0\}$.
Given a vertex $v \in V(G)$ we define its \emph{degree} as $d_G(v)\coloneqq|N_G(\{v\})|$.
We denote the minimum degree in a set of vertices as $\delta_G(A)\coloneqq\min\{d_G(v):v\in A\}$, and the maximum degree as $\Delta(G)\coloneqq\max\{d_G(v):v\in V(G)\}$.
We often consider the sequence of degrees of the vertices of $G$ ordered increasingly, and refer to it as the \emph{degree sequence} of $G$.

The binomial random graph $G_{n,p}$ is obtained by adding each of the edges of a complete graph on $n$ vertices with probability $p$, independently of the other edges.
We will always denote the vertex set of $G_{n,p}$ by $V$.
We use $G_{n,m,p}$ for a random bipartite graph with vertex classes of size $n$ and $m$, respectively; each edge between the classes is added with probability $p$ independently of every other edge, as above.
Whenever we consider a random bipartite graph between vertex sets $A$ and $B$, we also refer to this model as $G_{A,B,p}$.


\subsection{Tools for random graphs}

We will need the following Chernoff bound (see e.g.~\cite[Corollary 2.3]{JLR}).

\begin{lemma}\label{lem: chernoff}
Let $X$ be the sum of $n$ independent Bernoulli random variables and let $\mu \coloneqq \mathbb{E}[X]$.
Then, for all $0\le \delta \le 1$ we have that $\mathbb{P}[X  \ne  (1 \pm \delta) \mu] \leq 2e^{-\delta^2\mu/3}$.
\end{lemma}

The following lemmas are standard results for random graphs.
They can be proved using Chernoff bounds and the fact that the considered random variables follow binomial distributions.

\begin{lemma}\label{lem: edges}
There exist constants  $C, c>0$ such that for any $p \geq C\log n/n$ the random graph $G= G_{n,p}$ a.a.s.~satisfies that for all $X, Y \subseteq V$ we have
\[|e_G(X,Y) - |X||Y|p + |X\cap Y|^2p/2| \le c\sqrt{|X||Y|np}\]
and
\[|e_G'(X,Y) - |X||Y|p| \le c\sqrt{|X||Y|np}.\]
\end{lemma}

\begin{lemma}\label{lem: mindeg}
For every $\eta>0$, there exists a constant $C$ such that for $p \ge C\log n/n$ the random graph $G = G_{n,p}$ a.a.s.~satisfies that $d_G(v) = (1\pm\eta) np$ for all $v \in V$.
\end{lemma}

\begin{lemma}\label{lem: mindegGnm}
Let $A$ and $B$ be two disjoint sets of vertices with $|A|=n$, $|B|=m$ and $m=\Theta(n)$.
For every $\eta>0$, there exists a constant $C$ such that, for $p \ge C\log n/n$, the random graph $G=G_{A,B,p}$ a.a.s.~satisfies that for each $v \in A$ we have $d_G(v) = (1\pm\eta) mp$.
\end{lemma}

We now prove some properties of the subgraphs of the random graphs which satisfy the conditions of \cref{chvatal}.

\begin{proposition}\label{prop: neigh}
For every $0< \eps <1$, there exists $C>0$ such that for $p \ge C\log n /n$ the random graph $G=G_{n,p}$ a.a.s.~satisfies that, for all $H\in\mathcal{H}_{n,p}^{\eps,\eps}(G)$ and $G'\coloneqq G\setminus H$, the following hold:
\begin{enumerate}[label=(\roman*)]
    \item\label{neighitem1} For each $X \subseteq V$, we have 
    $|N_{G'} (X)| \ge  \min\{\eps |X| np/2, ~\eps n (\log n)^{-1/4}/2\}$.
    \item\label{neighitem3} For each $X\subseteq V$ with 
    $|X|\geq n(\log{n})^{-1/2}$,  we have that $|N_{G'}(X)|>(1-\epsilon^2/10)p^{-1} \delta_{G'}(X)$.
    In particular, $|N_{G'}(X)| \geq \eps n/2$.
    \item\label{neighitem4} $G'$ is connected.
\end{enumerate}
\end{proposition}

\begin{proof}
Choose a number $0<\eta \ll \epsilon$.
Consider the event that for all $v\in V$ we have
\begin{align}\label{equa:XYedges1.61}
d_{G}(v) = (1\pm \eta)np
\end{align}
and for all $X,Y\subseteq V$ with $|X|\geq n(\log{n})^{-1/2}$ and $|Y|\geq \eta n$ we have
\begin{align}\label{equa:XYedges1.62}
e'_{G}(X,Y) = (1 \pm \eta)|X||Y|p.
\end{align}
Throughout the proof, we condition on the event that \eqref{equa:XYedges1.61} and \eqref{equa:XYedges1.62} hold.
Note that \cref{lem: edges,lem: mindeg} imply that such an event a.a.s.~occurs.

\textbf{\ref{neighitem1}.} To prove \ref{neighitem1}, a simple calculation (see e.g.~\cite[Proposition 2.5(i)]{LS12}\COMMENT{Fix a set $X\subseteq V$ of size $|X|\leq\lceil(\log n)^{-1/4} p^{-1}\rceil$. 
For each $v\in V\setminus X$, let $Y_v$ be the indicator random variable of the event that $v\in N_G(X)$.
Then, \[\mathbb{P}[Y_v=1]=1-(1-p)^{|X|}=(1+o(1))|X|p\]
(the estimate follows from the fact that $|X|p=o(1)$).
Let $Y\coloneqq|N_G(X)|=\sum_{v\in V}Y_v$ and note that
\[\mathbb{E}[Y]=(1+o(1))|X|np.\]
Since the variables $Y_v$ are mutually independent, we can apply a Chernoff inequality to get \[\mathbb{P}[|Y-\mathbb{E}[Y]|\geq(\eps/4)\mathbb{E}[Y]]\leq e^{-\Omega_\eps(\mathbb{E}[Y])}.\]
Together with the previous estimate on $\mathbb{E}[Y]$, we conclude that
\[\mathbb{P}[Y\leq(1-\eps/2)|X|np]\leq e^{-\Omega_\eps(|X|np)}.\]
Since $np\geq C\log n$, $\mathbb{P}[Y\leq(1-\eps/2)|X|np]\leq n^{-C'|X|}$, where $C'=C'(\eps,C)$ can be made arbitrarily large by choosing $C$ appropriately.\\
Taking the union bound over all choices of $X$ we get
\[\sum_{1\leq|X|\leq(\log n)^{-1/4} p^{-1}}n^{-C'|X|}\leq\sum_{k=1}^n\binom{n}{k}n^{-C'k}\leq\sum_{k=1}^n\left(\frac{en}{k}n^{-C'}\right)^k=o(1),\]
which finishes the proof.}) shows that a.a.s.~for all $X \subseteq V$ of size at most $\lceil(\log n)^{-1/4} p^{-1}\rceil$,
\begin{equation}\label{equa:2.5.i}
    |N_G(X)| \ge (1-\eps/2) |X|np.
\end{equation}
As $H\in \mathcal{H}_{n,p}^{\eps,\eps}$, \eqref{equa:maxdegree} together with \eqref{equa:2.5.i} implies
\[|N_{G'}(X)| \ge |N_G(X)| - (1 - \eps)np |X| \ge \eps np |X|/2.\]
Given a set $X \subseteq V$ of size at least $(\log n)^{-1/4} p^{-1}$, we can choose a subset $X' \subseteq X$ of size $\lceil(\log n)^{-1/4} p^{-1}\rceil$, and apply the bound above to obtain $|N_{G'}(X)| \ge |N_{G'}(X')| \ge \eps n (\log n)^{-1/4}/2$.
This proves \ref{neighitem1}.

\textbf{\ref{neighitem3}.} As $H\in \mathcal{H}_{n,p}^{\eps,\eps}$, \eqref{equa:XYedges1.61} together with \eqref{equa:maxdegree} implies that $\delta_{G'}(X) \geq (\eps -\eta) np$.
For each $X\subseteq V$, we have
\begin{align}\label{equa:expansionlemma}
    e'_{G'}(X,V) &\geq |X|\delta_{G'}(X).
\end{align}
Suppose that there is a set $X\subseteq V$ with $|X|\geq n(\log{n})^{-1/2}$ and $|N_{G'}(X)| \leq (1-\eps^2/10) p^{-1} \delta_{G'}(X) $.
Let $Y\subseteq V$ be a set containing $N_{G'}(X)$ with $|Y|=(1-\eps^2/10) p^{-1}\delta_{G'}(X) \geq \eta n$\COMMENT{$|Y|=(1-\eps^2/10) p^{-1}\delta_{G'}(X)  \geq (1-\eps^2/10) (\eps-\eta) n  \geq \eta n$}.
Hence, \eqref{equa:XYedges1.62} implies that 
\[e'_{G}(X,Y) \leq (1 + \eta)p|X| (1-\eps^2/10) p^{-1}\delta_{G'}(X)
\leq (1- \eps^2/20) |X|  \delta_{G'}(X) \stackrel{\mathclap{\eqref{equa:expansionlemma}}}{<} e'_{G'}(X,V),\]
a contradiction to the fact that $N_{G'}(X)\subseteq Y$.
In particular, as  $\delta_{G'}(X) \geq (\epsilon -\eta) np$, we have 
$|N_{G'}(X)| \geq (1-\eps^2/10)(\eps-\eta) n \geq \epsilon n/2$.
This proves \ref{neighitem3}.

\textbf{\ref{neighitem4}.} Condition on the event that statements \ref{neighitem1} and \ref{neighitem3} hold, in addition to \eqref{equa:XYedges1.61} and \eqref{equa:XYedges1.62}.
Assume that $G'$ is not connected, and let $X \subseteq V$ be a (connected) component of $G'$ such that $|X|\leq n/2$.
Note that $|N_{G'}(X)| = |X|$.
As \ref{neighitem1} and \ref{neighitem3} both hold, it is easy to see that $|X|\geq \epsilon n/2$\COMMENT{Suppose otherwise.
Assume first that $|X|<n(\log n)^{-1/2}$.
Then, by \ref{neighitem1} we have that $|N_{G'} (X)| \ge  \min\{\eps |X| np/2, ~\eps n (\log n)^{-1/4}/2\}>|X|$, a contradiction to the fact that $X$ is a component of $G'$.
So we must have that $|X|\geq n(\log n)^{-1/2}$.
But then, by \ref{neighitem3} we have that $N_{G'}(X)\geq\eps n/2$, a contradiction.}.
Let $m\coloneqq |X|- \epsilon n/4 \geq \epsilon n/4$.

As $H\in \mathcal{H}_{n,p}^{\epsilon,\epsilon}$, by \cref{chvatal} there exists a labelling $v_1,\dots, v_n$ of $V$ with $d_H(v_1)\geq \ldots \geq d_H(v_n)$ such that we have either
\begin{equation}\label{eq: posa chvatal}
d_H(v_m) \leq (n-m)p - \eps np \qquad\text{ or }\qquad 
d_{H}(v_{n-m-\epsilon n}) \leq mp -\eps np.
\end{equation}
If the former is true, then there exists a set $X'\subseteq X\cap \{v_{m},\dots, v_{n}\}$ with $|X'| = \epsilon n/4$ and 
\[\delta_{G'}(X') \stackrel{\mathclap{\eqref{equa:XYedges1.61}}}{\geq} (1-\eta)np -  (n-m)p + \eps np \geq  mp + \eps np/2.\]
Then, \ref{neighitem3} ensures that $|N_{G'}(X')|\geq (1-\epsilon^2/10)(m+ \epsilon n/2)\geq m + \epsilon n/3 >|X|$, a contradiction to the fact that $X$ is a component of $G'$.

Hence, we may assume that the latter of \eqref{eq: posa chvatal} holds.
In this case, there are at least $m+ \epsilon n \geq |X|+\eps n/2$ vertices $v$ with $d_{H}(v) \leq mp - \eps np$, hence there exists a set $Y\subseteq\{v_{n-m-\epsilon n},\dots, v_{n}\}\setminus X$ with $|Y|\geq\eps n/2$ and
\[\delta_{G'}(Y) \stackrel{\mathclap{\eqref{equa:XYedges1.61}}}{\geq} (1-\eta)np -  mp + \eps np \geq  (n-m) p + \eps np/2.\]
Then, \ref{neighitem3} ensures that $|N_{G'}(V\setminus X)| \geq |N_{G'}(Y)| \geq (1-\epsilon^2/10)(n-m+ \epsilon n/2)\geq n-m + \epsilon n/3 >|V\setminus X|$, a contradiction to the fact that $X$ is a component of $G'$.
\end{proof}


\section{Chv\'atal-type resilience for matchings in random graphs}

\begin{proof}[Proof of \cref{teor:ChvMatch}]
Let $0<1/n \ll  1/C \ll\eta\ll\eps\ll1$ and $1/c <1$, where $n$ is even and $c$ is the constant given by \cref{lem: edges}.
We condition on the event that $G=G_{n,p}$ satisfies the assertions of \cref{lem: edges}, \cref{lem: mindeg} and \cref{prop: neigh} with the chosen constants $\eps$, $\eta$, $C$ and $c$, which happens a.a.s. 
We will show that all such $G$ are $(\eps,\eps)$-Chv\'atal-resilient with respect to containing a perfect matching.
Let $H\in\mathcal{H}^{\eps,\eps}_{n,p}(G)$ and let $G'\coloneqq G\setminus H$.
Let $v_1,\ldots,v_n$ be an ordering of the vertices as in \cref{chvatal}.
Let $D(H)\coloneqq\{ v_{ \lceil n/2\rceil} ,\dots, v_n\}$.
In particular, by \cref{lem: mindeg} we have that
\begin{equation}\label{equa:mindegMatch2}
    \delta_{G'}(D(H))\geq (1+\eps)np/2.
\end{equation}
By Tutte's theorem, it suffices to show that, for any vertex set $U\subseteq V$, the number of odd components of $G'-U$ is at most $|U|$ (here a component is odd if it contains an odd number of vertices).
As we conditioned on the assertion of \cref{prop: neigh}\ref{neighitem4} and since $n$ is even, this holds if $U$ is the empty set.

Hence, we will prove that, for any non-empty $U\subseteq V$, the number of (not necessarily odd) components of $G'-U$ is at most $|U|$.
As each component of $G'-U$ has at least one vertex, we may further assume that $|U|<n/2$.

Let $U\subseteq V$ with $|U|< n/2$ and let $k$ be the total number of components of $G'-U$.
To derive a contradiction, assume that $k>|U|$; in particular, $k\geq 2$.
Enumerate the components in $G'-U$ as $C_1,\ldots,C_k$ with $|C_1|\leq |C_2| \leq \ldots \leq |C_k|$.
For each $S\subseteq[k]$, let $C_S\coloneqq\bigcup_{i\in S}C_i$.
We consider the cases where $|U|$ is small and large separately.

\vspace{8 pt}

\noindent  \textbf{Case 1:} $|U| \le \eps n/10$.

First, we prove that $|C_k|$ is large in this case.
\begin{claim}
We have $|C_k| > n/2$.
\end{claim}

\begin{proof}
Suppose otherwise that $|C_k|\leq n/2$.
Let 
\[\mathcal{S}\coloneqq\{S\subseteq[k]:|C_S\cap D(H)|\geq\eps n\}.\] 
Let $S^*\in\mathcal{S}$ be a set in $\mathcal{S}$ with the minimum $|C_{S^{*}}|$.
We claim that $|C_{S^*}| \le n/2$.
Indeed, suppose this is not the case.
Then, we have $|S^*|\geq 2$.
As a partition of $S^*$ into two non-empty sets yields two disjoint sets not in $\mathcal{S}$, we have $|C_{S^*}\cap D(H)|<2\eps n$.
Thus $C_{[k]\setminus S^*}$ satisfies that $|C_{[k]\setminus S^*}|\leq n/2$ and $|C_{[k]\setminus S^*}\cap D(H)|\geq n/2-3\eps n$\COMMENT{$2\eps n$ are in $C_{S^*}$ and a few more may be in $U$.} so we have $[k]\setminus{S^*}\in\mathcal{S}$, which contradicts the minimality of $C_{S^*}$. Hence we have $|C_{S^*}| \le n/2$.

Let $D\coloneqq C_{S^*}\cap D(H)$.
As we have $|D|\geq\eps n$, by \eqref{equa:mindegMatch2} and \cref{prop: neigh}\ref{neighitem3} we have
\[|N_{G'}(D)|>(1-\eps^2/10) (1+\eps) n/2 >n/2+|U|.\]
It follows that at least one vertex $v\in D\subseteq C_{S^*}$ is adjacent to a vertex $u\in C_{[k]\setminus S^*}$, a contradiction.
This proves the claim.
\end{proof}

Let $\ell\coloneqq|C_{[k-1]}|$.
Note that $\ell<n/2$.

\begin{claim}\label{claim: ChvMatch2}
We have $\ell < \eps n/6$. 
\end{claim}

\begin{proof}
Assume otherwise that $\ell \ge\eps n/6$.

First, assume that $H$ satisfies \eqref{equa:Posadegree} for all $i\in [\ell]\setminus [\ell-\eps n/8]$\COMMENT{i.e.~P\'osa degree sequence}.
Note that the set $C' \coloneqq C_{[k-1]} \setminus \{v_1,\dots, v_{\ell-\eps n/8}\}$ satisfies $|C'|\geq\eps n/8$.
Because $G$ satisfies the assertion of \cref{lem: mindeg} and $v_{\ell-\eps n/8+1}$ satisfies \eqref{equa:Posadegree} for $H$, we have
\[\delta_{G'}(C')\geq \delta_{G}(V)-d_H(v_{\ell-\eps n/8+1}) \geq (1-\eta)np - ((n- \ell+\eps n/8-1)p - \eps np) \geq \ell p+3\eps np/4.\]
As $G'$ satisfies the assertion of \cref{prop: neigh}\ref{neighitem3}, we have
\[|N_{G'}(C_{[k-1]})|\geq|N_{G'}(C')|\geq(1-\eps^2/10)(\ell+ 3\eps n/4)\geq
\ell+\eps n/2>|C_{[k-1]}|+|U|,\]
a contradiction as $C_k$ and $C_{[k-1]}$ are disconnected in $G'-U$.

So suppose that there is an index $j\in[\ell]\setminus[\ell-\eps n/8]$ such that $H$ does not satisfy \eqref{equa:Posadegree} for $j$\COMMENT{i.e.~we have a Chv\'atal degree sequence}.
We have that the set $C'' \coloneqq C_k\setminus \{v_1,\dots, v_{ n-j-\eps n }\}$ satisfies 
\[|C''| \geq |C_k| - (n-j-\eps n) = n - \ell - |U| -(n-j-\eps n) \geq \eps n/4.\]
Here, we obtain the final inequality as $|U| \leq \eps n/10$ and $ j\geq \ell -\eps n/8$.
Moreover, because $G$ satisfies the assertion of \cref{lem: mindeg}, the fact that \eqref{equa:Chvataldegree} holds for $j$ implies that 
\[\delta_{G'}(C'') \geq \delta_{G}(V)-d_H(v_{n-j-\eps n+1}) \geq (1-\eta)np - (j -\eps n)p \geq
(n-j+\eps n/2)p.\]
As $G'$ satisfies the assertion of \cref{prop: neigh}\ref{neighitem3}, this shows that
\[|N_{G'}(C'')|> (1-\eps^2/10)(n-j+\eps n/2) \geq n-\ell+\eps n/6>|C_k|+|U|,\]
a contradiction to the fact that $C_k$ is a component of $G'-U$. 
This proves the claim.
\end{proof}

It follows from the previous two claims that $G'-U$ has one `giant' component $C_k$, containing more than $(1-\eps/3)n$ vertices. 
The following claim will give us the desired contradiction.

\begin{claim}\label{claim: ChvMatch3}
For any set $W \subseteq V$ with $|W| < \varepsilon n/6$, we have that $|N_{G'}(W)| > 2|W|$.
\end{claim}

\begin{proof}
If $|W|\leq n(\log n)^{-1/2}$, then, as $G'$ satisfies the assertion of \cref{prop: neigh}\ref{neighitem1}, we have
\[|N_{G'}(W)|\geq \min\bigg\{\frac{1}{2}\eps |W| np, ~\frac{1}{2}\eps n (\log n)^{-1/4}\bigg\} > 2|W|.\]
If we have $n(\log n)^{-1/2} \le |W| < \eps n/6$, then, because $G'$ satisfies the assertion of \cref{prop: neigh}\ref{neighitem3}, we have
\[|N_{G'}(W)|\geq\eps n/2>2|W|.\qedhere\]
\end{proof}

Recall that $|U|\leq k-1 \leq \ell$.
As $C_k$ and $C_{[k-1]}$ are disconnected in $G'-U$, we have $|N_{G'}(C_{[k-1]})| \leq |C_{[k-1]}|+|U| \leq 2 \ell$.
However, by \cref{claim: ChvMatch2,claim: ChvMatch3}, we have $|N_{G'}(C_{[k-1]})|> 2\ell$, a contradiction.
This concludes Case 1. 

\vspace{8 pt}

\noindent \textbf{Case 2:} $|U|> \eps n/10$.

Let $S\coloneqq \{i \in [k]: |C_i| < 2\sqrt{n}\}$ and $t \coloneqq|U|$.
We first claim that 
\begin{equation}\label{equa:C_Ssize}
    t  - \sqrt{n} \le |C_S|.
\end{equation}
Indeed, suppose otherwise. 
As $k>t$ and each component of $G'-U$ contains at least one vertex, we have 
\[|[k] \setminus S|>t-|S|\geq t - |C_S| > \sqrt{n}.\]
Hence,  $|C_{[k]\setminus S}| \geq 2\sqrt{n} \cdot |[k] \setminus S|>2n$, a contradiction. 
Thus $t  - \sqrt{n} \le |C_S|$.

As $G$ satisfies the assertion of \cref{lem: edges}, by the definition of $S$ we have
\begin{align}\label{equa:edgesincomponents}
    e_{G'}(C_S) &\leq \sum_{i\in S} e_{G}(C_i) 
    \leq  \sum_{i\in S} \left( |C_i|^2 p + c |C_i| \sqrt{np}\right) 
    \leq\left( \sum_{i\in S} |C_i| \right) ( 2\sqrt{n}p + c\sqrt{np} ) \nonumber\\
    & \leq |C_S| \cdot  4 c \sqrt{np} \leq 4 c n^{3/2} p^{1/2} \leq \eta n^2p.
\end{align}
We also claim that 
\begin{equation}\label{eq: CS not contian}
\begin{minipage}[c]{0.9\textwidth}
$C_S$ does not contain any set $C'$ with $|C'|\geq \eps n/20$ and $\delta_{G'}(C') \geq tp+ \eps np/2$. 
\end{minipage}
\end{equation}
Indeed, suppose $C_S$ contains such a set $C'$.
By \eqref{equa:edgesincomponents} we have that
\COMMENT{We have
\begin{align*}
    e_{G'}(C',U) &\geq|C'|\delta_{G'}(C')-2e_{G'}(C_S)\geq|C'|(|U|p+\eps np/2)-2\eta n^2p \\
    &\geq \eps n|U|p/20 + \eps^2 n^2 p/40-2\eta n^2p > \eps n|U|p/20 + \eps^2 n^2 p/50.
\end{align*}
Note that the first inequality holds since, by assumption, there can be no edges to vertices in connected components other than those indexed by $S$.}
\[e_{G'}(C',U)\geq|C'|\delta_{G'}(C')-2e_{G'}(C_S) > |C'| t p+ \eps^2 n^2 p/50.\] 
On the other hand, as $G$ satisfies the assertion of \cref{lem: edges}, we have
\[e_{G'}(C',U)\leq e_{G}(C',U) \leq |C'|tp + c\sqrt{|C'|tnp}\leq|C'|tp +\eps^2 n^2 p/100,\]
a contradiction. Hence, such a set $C'$ does not exist.

Suppose that $H$ satisfies  \eqref{equa:Posadegree} for all $i\in [t]\setminus[t-\eps n/10]$.
As $G$ satisfies the assertion of \cref{lem: mindeg} and by \eqref{equa:C_Ssize}, the set $C'' \coloneqq C_S \setminus \{v_1,\dots, v_{t-\eps n/10} \}$ satisfies $|C''| \geq |C_S| -t + \eps n/10 \geq \eps n/20$ and 
\[\delta_{G'}(C'') \geq (1-\eta)np - (n-t - 9\eps n/10)p \geq
tp + \eps np/2,\]
a contradiction to \eqref{eq: CS not contian}.

Hence, there exists $j\in [t]\setminus[t-\eps n/10]$ such that $H$ does not satisfy \eqref{equa:Posadegree} for $j$.
By \eqref{equa:Chvataldegree} and \cref{lem: mindeg}, this means that 
\[\delta_G(V)-d_H(v_{n-t -9\eps n/10-1})\geq (1-\eta)np - (j -\eps n)p \geq (n- t+4\eps n/5)p.\]
Therefore, the set $R \coloneqq V\setminus (U\cup \{v_1,\dots,v_{n-t-9\eps n/10}\})$ satisfies $|R|\geq 9\eps n/10$ and $\delta_{G'}(R) >  (n- t+4\eps n/5)p.$ As $t\leq n/2$, we have $\delta_{G'}(R) \geq tp + \eps np/2$.
Hence, we conclude $|R\cap C_{S}|  < \eps n/20$, otherwise we have a contradiction to \eqref{eq: CS not contian}.

Hence, $R'\coloneqq R\cap C_{[k]\setminus S}$ satisfies $|R'|\geq 4\eps n/5$.
As $G'$ satisfies the assertion of \cref{prop: neigh}\ref{neighitem3}, we conclude that 
\[|N_{G'}(R')|\geq (1-\eps^2/10) ( n-t + 4\eps n/5) \geq n-t + \eps n/2\stackrel{\mathclap{\eqref{equa:C_Ssize}}}{>}|V\setminus C_S|.\]
This is a contradiction as $R'$ lies inside $C_{[k]\setminus S}$, which is disconnected from $C_S$ in $G'-U$.
\end{proof}

We now show that \cref{teor:ChvMatch} is best possible in the sense that $(\eps,\eps)$-Chv\'atal-resilience cannot be improved to allow for $(\eps,(3np)^{-1})$-Chv\'atal-resilience.
That is, unlike the classical theorem of Chv\'atal, the random graphs analogue requires an extra shift in the indices whenever we veer from a P\'osa degree sequence. 

Given an $n$-vertex graph $G$, we say that $G$ contains an \emph{optimal matching} if it has a matching of size $\lfloor n/2 \rfloor$. 
In particular, if $G$ does not contain an optimal matching, then $G$ cannot be Hamiltonian. 
Note that \cref{teor:Counter} implies \cref{teor:CounterIntro}.

\begin{theorem}\label{teor:Counter}
For every $0< \eps < 10^{-6}$ there exists $C>0$ such that, for any $C\log n/n \leq p\leq 1/25$, \COMMENT{This bound on $p$ comes from the following observations.
On the one hand, we want $|Y|\geq2$ in order to show that the example does not satisfy Chv\'atal in the classical, non-shifted sense.
Imposing this, we conclude that $p\leq1/(4(1+\eta))$.
On the other hand, for the sake of simplicity in the statement, we want to have that $\lfloor1/((1+\eta)2p)\rfloor\geq3+\lceil1/(3p)\rceil$ for all $p$ in our range (see the end of the proof).
In particular, this fails around $p=1/(12(1+\eta))$, but works for all $p\leq1/20$ (for small values of $\eta$). 
We have not tried to optimise this constant, nor the upper bound on $\eps$.} the random graph $G=G_{n,p}$ is a.a.s.~not $(\eps,\lceil(3p)^{-1}\rceil/n)$-Chv\'atal-resilient with respect to containing an optimal matching.
\end{theorem}

The proof strategy is as follows. 
We consider $G_{n,p}$ and remove appropriate edges to create a graph $G'$ having an independent set $X$ with $|N_{G'}(X)|<|X|-1$. 
This ensures that $G'$ does not contain an optimal matching.
We conclude the proof by showing that $G\setminus G' \in \mathcal{H}^{\eps, \lceil(3p)^{-1}\rceil/n}_{n,p}$.

\begin{proof}
Let $1/n \ll 1/C \ll \eta \ll \eps < 10^{-6}$.
Let $Y\subseteq V$ be any set of vertices of size $\lfloor ((1+\eta)2p)^{-1}\rfloor$.
Now expose all edges of $G$ incident to $Y$. 
Let $\mathcal{E}_1$ be the event that, for each vertex $y\in Y$, we have 
\begin{equation}\label{equa:ChvCounterMinDeg}
    d_G(y)=(1\pm\eta)np.
\end{equation}
Note that \cref{lem: mindeg} implies that $\mathcal{E}_1$ happens a.a.s. We condition on the event $\mathcal{E}_1$.
Thus we have
\[|N_G(Y)|\leq\sum_{y\in Y}d_G(y)\leq|Y|(1+\eta)np\leq n/2.\]

Fix disjoint sets $X,U\subseteq V\setminus(Y\cup N_G(Y))$ with $|X|=100\eps n$ and $|U|=|X|-2$.
Now, expose all remaining edges of $G$ (i.e.~those not incident to $Y$).
Let $\mathcal{E}_2$ be the event that the following hold for all $v\in V\setminus Y$ and $Z\in \{X,U\}$:
\begin{align}
e_{G}(v, Z) &= (1\pm \eta)|Z|p,\label{equa:ChvCounter}\\
d_{G}(v) &= (1\pm \eta)np.\label{equa:ChvCounter1111}
\end{align}
By \cref{lem: mindeg,lem: mindegGnm}, the event $\mathcal{E}_2$ happens a.a.s.~under conditioning on $\mathcal{E}_1$\COMMENT{Here, for \eqref{equa:ChvCounter1111} we are implicitly using the fact that the conditioning on $\mathcal{E}_1$ only affects $|Y|=o(n)$ possible edges incident to each $v$.}. 
We condition on the event that both $\mathcal{E}_1$ and $\mathcal{E}_2$ hold, i.e.~that $G$ satisfies \eqref{equa:ChvCounterMinDeg}--\eqref{equa:ChvCounter1111}.
We will show that every such $G$ is not $(\eps,\lceil(3p)^{-1}\rceil/n)$-Chv\'atal-resilient with respect to containing an optimal matching.

We construct a spanning subgraph $G'$ of $G$ by deleting all edges in $G[X]$ and all edges in $G[X,V\setminus(X\cup U)]$.
From the construction, $X$ is an independent set of $G'$ and $N_{G'}(X)\subseteq U$.
Thus, $|N_{G'}(X)| \leq |U| <|X|-1$, hence $G'$ does not contain an optimal matching. 

Let $\gamma\coloneqq\lceil(3p)^{-1}\rceil/n$.
Now it suffices to show that $H\coloneqq G\setminus G' \in \mathcal{H}_{n,p}^{\eps,\gamma}$. 
From the construction, it is easy to see that \eqref{equa:ChvCounter} and \eqref{equa:ChvCounter1111} imply that, for all $u\in Y\cup U$, $x\in X$ and $v\in V\setminus(X\cup U\cup Y)$,
\begin{align}\label{eq: H degree 1}
d_{H}(u) =0, \quad d_{H}(x) = (1-100\varepsilon\pm3\eta)np \quad \text{ and } \quad d_{H}(v) = (1\pm\eta)100\eps np.
\end{align}
Let $v_1,\dots, v_n$ be an ordering of $V$ with $d_H(v_1)\geq \ldots \geq d_H(v_n)$.
Observe that, in this ordering, $X=\{v_1,\ldots,v_{|X|}\}$ and $U\cup Y=\{v_{n-|U\cup Y|+1},\ldots,v_n\}$.

We now show that $H\in\mathcal{H}^{\eps,\gamma}_{n,p}$.
As \eqref{eq: H degree 1} implies $\Delta(H)\leq (1-99\eps)np$, $H$ satisfies \eqref{equa:Posadegree} for all $i\in [98\eps n]$. 
Note that for each $i \in [100\eps n]\setminus [98\eps n]$, \eqref{eq: H degree 1} implies that\COMMENT{We need $|Y|\geq(3p)^{-1}+3$.} $d_H(v_{n-i-\gamma n})=0 \leq (i-\eps n)p$. 
Thus $H$ satisfies \eqref{equa:Chvataldegree} for all $i\in [100\eps n]\setminus [98\eps n]$. 
Finally, for $i\in[n/2-1]\setminus[100\eps n]$, \eqref{eq: H degree 1} implies that
$d_H(v_i) \leq 101 \eps np\leq (n-i)p - \eps np$, where the final inequality holds with room to spare. 
Thus $H$ satisfies \eqref{equa:Posadegree} for all $i\in [n/2-1]\setminus [100 \eps n]$. 
Hence, $H\in \mathcal{H}_{n,p}^{\eps,\gamma}$.
Therefore, $G_{n,p}$ a.a.s.~contains a subgraph $H\in \mathcal{H}_{n,p}^{\eps,\gamma}$ such that $G_{n,p}\setminus H$ does not contain an optimal matching.
\end{proof}


\section{P\'osa's theorem for Hamilton cycles in random graphs}

Our approach for the proof of \cref{teor:posa} builds on the ideas of \citet{LS12},
with some modifications and additional steps to account for the increased
flexibility in the choice of the graph $H$ that we remove.
Thus we only describe the necessary tools as well as the main steps.
The corresponding proofs that we omit here can be found in the appendix.
For $H\in\mathcal{H}^{\eps}_{n,p}$, we rely heavily on the fact that graphs of the form $G_{n,p}\setminus H$ have good expansion properties; namely, they satisfy \cref{prop: neigh}.

Whenever we consider a path $P$ on a vertex set $W$ we mean that $V(P)\subseteq W$.
Let $G$ be a graph and let $P=v_1\ldots v_\ell$ be a path on $V(G)$.
Let $v\coloneqq v_1$ and $u\coloneqq v_\ell$ be the endpoints of $P$.
Suppose $v_i \in N_G(v)$ for some $i \ne \ell$.
Then, we can also consider the path $P'=v_{i-1}v_{i-2}\ldots vv_iv_{i+1}\ldots u$ in $G\cup P$.
We refer to the path $P'$ as a \emph{rotation} of $P$ within $G$ with \emph{fixed endpoint} $u$ and \emph{pivot} $v_i$.
We call $v_{i-1}v_i$ the \emph{broken edge} of the rotation.

Starting from $P$, we will consider successive rotations of $P$ to obtain new paths, always leaving one of the endpoints of $P$ fixed.
We only consider rotations whose broken edges are edges in the original path $P$.

For any vertex $x \in V(P)$, let $x_{P,u}^{-}$ and $x_{P,u}^{+}$ denote the predecessor and successor of $x$ along $P$, respectively (where $P$ is oriented towards the fixed endpoint $u$).
Similarly, given any set $X\subseteq V(P)$, we denote $X_{P,u}^+\coloneqq\{x_{P,u}^+:x\in X\}$ and $X_{P,u}^-\coloneqq\{x_{P,u}^-:x\in X\}$.

Let $\mathcal{R}_{G,P,u}\subseteq V(P)$ be the set of all vertices $x\in V(P)$ such that there exists a path $P_x$ in $G\cup P$ with endpoints $u$ and $x$ which can be obtained by taking successive rotations of $P$ within $G$ with fixed endpoint $u$.
(As mentioned before, we only consider rotations whose broken edges are in $P$.)
Whenever we consider a vertex $x\in\mathcal{R}_{G,P,u}$, the notation $P_x$ will be used to denote a path with endpoints $x$ and $u$ which can be obtained by the minimum number of rotations of $P$ (whenever there is more than one choice for $P_x$, we fix such a choice arbitrarily among all the possibilities).\COMMENT{I think this is important and overlooked in \cite{LS12}. If $P_x$ is instead allowed to vary among the different paths (that is, we consider different $P_x$ for the same $x$) that fit the definition, then we do different damage to different path segments in \cref{claim: rotations3}. So I think it is best to make it clear here that we fix a $P_x$ and never consider other possible sequences of rotations.}
Let $R_{G,P,u}^0\coloneqq\{v\}$ and $R_{G,P,u}^t$ be the set of vertices $x\in\mathcal{R}_{G,P,u}$ such that $P_x$ is obtained by at most $t$ rotations.

Given any set $A\subseteq\mathcal{R}_{G,P,u}$, we denote by $R_{G,P,u}(A)$ the union of $A$ and the set of endpoints of all paths which are obtained via a single rotation of $P_a$ with $u$ as a fixed endpoint, for any $a\in A$.

The following observation is well-known.
We include the short proof in the appendix.

\begin{lemma}\label{claim: rotations1}
Let $G$ be a graph.
Let $P'$ be a path on $V(G)$ and let $P=v_1\ldots v_\ell$ be a longest path in $G\cup P'$.
Then, for all $t\geq 0$ we have 
\[|R_{G,P,v_\ell}^{t+1}| \ge \frac{1}{2}(|N_{G}(R_{G,P,v_\ell}^t)|-3|R_{G,P,v_\ell}^t|).\]
\end{lemma}

Next, we restrict ourselves to the random graph $G_{n,p}.$ 
Given a `large' set $A$ of endpoints obtainable via a `small' number of successive rotations of a longest path $P$, we prove a lower bound on the number of endpoints obtainable from $A$ via one further rotation.

\begin{lemma}\label{claim: rotations3}
Let $0<1/C\ll\eta \ll \eps < 1$.
For $p \geq C\log{n}/n$, the random graph $G=G_{n,p}$ a.a.s.~satisfies the following.
Let $G'$ be a subgraph of $G$ and $P'$ be a path on $V$.
Let $P=v_1\ldots v_\ell$ be a longest path in $G'\cup P'$. 
Then, for all $A\subseteq R_{G',P,v_\ell}^{\eta \log n}$ with $|A|\geq \eps n/100$, we have that $|R_{G',P,v_\ell}(A)|\geq p^{-1}\delta_{G'}(A) - \eps n/10$.
\end{lemma}

The proof of \cref{claim: rotations3} is similar to (part of) the proof of Lemma~3.2 in \cite{LS12}.
For completeness, we include the details in the appendix.

We now combine the two previous results to give a lower bound on the number of endpoints which can be generated via successive rotations of a path $P$ with one fixed endpoint. 

\begin{lemma}\label{lem: re}
Let $0<1/C \ll \eps <1$. 
For $p \geq C\log n/n$, the random graph $G=G_{n,p}$ a.a.s. satisfies the following.
Let $H\in \mathcal{H}_{n,p}^{\eps}(G)$ and $G'\coloneqq G\setminus H$.
Let $P'$ be a path on $V$.
For any longest path $P=v_1\ldots v_\ell$ in $G'\cup P'$, there exists $U\subseteq V$ with $|U|\geq(1/2 + \eps/4)n$ such that, for every $v\in U$, there exists a longest path $Q_v$ in $G'\cup P'$ with endpoints $v_\ell$ and $v$, where $V(Q_v)=V(P)$.
\end{lemma}

\begin{proof}
Let $u\coloneqq v_\ell$.
Throughout this proof we write $R^t$ for $R_{G',P,u}^t$ and $R(A)\coloneqq R_{G',P,u}(A)$ for any $A\subseteq\mathcal{R}_{G',P,u}$.
Let $\eta$ be a number such that $1/C\ll\eta \ll \eps$. 
Condition on the event that the following holds for all $v\in V$:
\begin{align}\label{equa:degde}
d_{G}(v) = (1\pm \eta)np.
\end{align}
We also condition on the event that the assertions of \cref{prop: neigh} and \cref{claim: rotations3} hold for $G$.
By \cref{lem: mindeg,claim: rotations3,prop: neigh}, each of these events holds a.a.s.

Note that \eqref{equa:degde} and the fact that $H\in \mathcal{H}_{n,p}^{\eps}(G)$ imply that, for any set $X\subseteq V$  with $|X|\geq \eps n/10$,
\begin{equation}\label{eq: delta contains}
\begin{minipage}[c]{0.9\textwidth}
 there exists a set $X'\subseteq X$ with $|X'|\geq \eps n/20$ and $\delta_{G'}(X') \geq \min\{|X|, n/2 \}p+\eps np/2$.
\end{minipage}
\end{equation}

Note that, since $P$ is a longest path in $G'\cup P'$, we have that $N_{G'}(x) \subseteq V(P)$ for all $x\in\mathcal{R}_{G',P,u}$.
We will consider successive rotations of $P$, keeping $u$ fixed, to derive a lower bound on the number of distinct endpoints of different longest paths in $G'\cup P'$ with an endpoint $u$.

By \cref{claim: rotations1} together with the assertion of \cref{prop: neigh}\ref{neighitem1}, for each $t\geq0$,  we have
\[|R^{t+1} | \geq \frac{1}{2}\left( \min\left\{\frac{1}{2} \eps|R^t|np, \frac{1}{2}\eps n (\log{n})^{-1/4}\right\} - 3|R^t|\right).\]
As $R^0 =\{v_1\}$ and $\eps np/2 > \log{n}$, the above inequality implies that there exists $s \in \mathbb{N}$ with $s \leq \frac{1}{2} \eta \log{n}$ such that
\begin{equation}\label{equa:rotationsZtsize}
    |R^{s}| \ge \frac{\eps n}{ 5(\log{n})^{1/4}}. 
\end{equation}\COMMENT{Without loss of generality we can consider a subset of size $\frac{\eps n}{100(\log{n})^{1/4}}$. Note that $\frac{1}{4} ( \eps n (\log n)^{-1/4} ) - \frac{3\eps n}{ 100(\log{n})^{1/4}}\geq \frac{\eps n}{ 5(\log{n})^{1/4}}$, and ${\log n}^ {\log n} > n$.}
Again, by applying \cref{claim: rotations1}  together with the assertion of \cref{prop: neigh}\ref{neighitem3}, we obtain that 
$|R^{s+1}| \geq \eps n/10$.\COMMENT{Note that $\frac{1}{2}( \eps n/2 - 3\cdot \eps n/10 ) \geq \eps n/10$.}

Now, in order to show that $|R^{s+ 5\eps^{-1}+1}| \geq (1/2+\eps/4)n$, we  will iteratively construct sets $Y_0,\dots, Y_{5\eps^{-1}}$ as follows.

Let $Y_0\coloneqq R^{s+1}$. 
Suppose that for some $0\leq j < 5\eps^{-1}$ we have already constructed $Y_j$ with $|Y_j|\geq (j+1) \eps n/10$.
We use \eqref{eq: delta contains} to obtain a subset $Y'_j \subseteq Y_j$ with $|Y'_j| \geq  \eps n/20$ and $\delta_{G'}(Y'_j) \geq  (j+1) \eps np/10 + \eps np/2$. 
Let $Y_{j+1}\coloneqq R(Y'_j)$.
By \cref{claim: rotations3}, we have 
\[|Y_{j+1}| = |R(Y'_j)| \geq p^{-1}\delta_{G'}(Y'_j) - \eps n/10 
\geq (j+1)\eps n/10 + \eps n/2 - \eps n/10 \geq (j+2)\eps n/10.\]
Note that we can apply \cref{claim: rotations3} as $s+1+j\leq s+ 5\eps^{-1}\leq \frac{1}{2}\eta \log{n} + 5\eps^{-1} \leq \eta \log{n}$.
By repeating this for $0\leq j < 5\eps^{-1}$, we have 
$|Y_{5\eps^{-1}}| \geq (1/2 + \eps/4)n$\COMMENT{To get this, in the last step we use the second-to-last statement in the previous general formula.}.

By the construction, $Y_{5\eps^{-1}} \subseteq R^{s+5\eps^{-1}+1} \subseteq R^{\eta \log{n}}$. Letting $U\coloneqq R^{\eta \log{n}}$ concludes the proof.
\end{proof}

\begin{definition}\label{REdelta}
Let $\delta>0$. 
We say that a connected $n$-vertex graph $G$ has property $RE(\delta)$ if one of the following holds for every path $P$ on $V(G)$:
\begin{enumerate}[label=(\roman*)]
    \item there exists a path longer than $P$ in the graph $G \cup P$,
    \item there exists $S_P\subseteq V(G)$ with $|S_P| \ge \delta n$ and a collection $\{T_v : v\in S_P\}$ of subsets of $V(G)$ with $|T_v|\geq \delta n$ for all $v\in S_P$ satisfying the following: for all $v\in S_P$ and $w \in T_v$, the graph $G\cup P$ contains a path $Q$ between $v$ and $w$ with $V(Q)=V(P)$.
\end{enumerate}
\end{definition}

\begin{lemma}\label{lem: re2}
For every $0< \eps <1 $ there exists $C>0$ such that, for $p \geq C\log n/n$, the random graph $G=G_{n,p}$ a.a.s.~satisfies the following.
Let $H\in \mathcal{H}_{n,p}^{\eps}(G)$ and $G'\coloneqq G\setminus H$.
Then, $G'$ satisfies $RE(1/2 + \eps/4)$.
\end{lemma}

\begin{proof}
Recall that $G$ a.a.s.~satisfies the assertions of \cref{prop: neigh} and \cref{lem: re}.
We prove that $G'$ satisfies $RE(1/2 + \eps/4)$ conditioned on this.

By \cref{prop: neigh}\ref{neighitem4}, $G'$ is connected.
Let $P$ be any path on $V$.
We may assume that $G'\cup P$ does not contain a path which is longer than $P$.
Let one of the endpoints of $P$ be $u$. 
By \cref{lem: re}, there exists $S_P\subseteq V$ with $|S_P|\geq(1/2 + \eps/4)n$ and such that, for every $v\in S_P$, there exists a path $Q_v\subseteq G'\cup P$ with endpoints $u$ and $v$ such that $V(Q_v) = V(P)$. 
For each path $Q_v$ we can fix $v$ and apply \cref{lem: re} again to obtain a set $T_v\subseteq V$ such that $|T_v|\geq(1/2 + \eps/4)n$ and for every $x\in T_v$ there is a path $Q_{xv}\subseteq G'\cup P$ from $x$ to $v$ with $V(Q_{xv})=V(P)$.
The result follows.
\end{proof}

\begin{definition}\label{def: comp}
Let $\delta>0$ and let $G_1$ be a graph on $n$ vertices with property $RE(\delta)$.
We say that a graph $G_2$ with $V(G_2)=V(G_1)$ complements $G_1$ if, for every path $P$ on $V(G_1)$, one of the following holds:
\begin{enumerate}[label=(\roman*)]
\item\label{def46condition1} there exists a path longer than $P$ in $G_1 \cup P$,
\item\label{def46condition2} there exist sets $S_P$ and $T_v$ as in \cref{REdelta} and vertices $v \in S_P$ and $w \in T_v$ such that $vw$ is an edge of $G_1 \cup G_2$. 
\end{enumerate}
\end{definition}

\begin{proposition}[\cite{LS12}]\label{prop: comp}
Let $\delta>0$.
For every $G_1 \in RE(\delta)$ and $G_2$ complementing $G_1$, the union $G_1\cup G_2$ is Hamiltonian.
\end{proposition}

Finally, we state two lemmas which are used to complete the proof of \cref{teor:posa}. 
The first says that, given $G = G_{n,p}$ and $H \in \mathcal{H}^\eps_{n,p}(G)$, the graph $G\setminus H$ complements every `small' subgraph of $G$ which has property $RE(1/2 + \eps/4)$. 
The final lemma then says that $G'$ actually contains some such `small' graph as a subgraph.
We include the proof of \cref{lem: ext} in the appendix.

\begin{lemma}\label{lem: ext}
For every $0<\eps <1$, there exist $C, \delta >0$ such that for $p \geq C\log n/n$ we have that $G = G_{n,p}$ a.a.s.~satisfies the following property:
for any $H \in \mathcal{H}^\eps_{n,p}(G)$, the graph $G\setminus H$ complements all graphs $R\subseteq G$ which satisfy $RE(1/2 + \eps/4)$ and have at most $\delta n^2 p$ edges.
\end{lemma}

\begin{lemma} \label{lem: exist}
For all $0<\eps, \delta \le 1$, there exists $C >0$ such that, for $p \ge C\log n/n$, the graph $G= G_{n,p}$ a.a.s.~satisfies the following property.
Let $H\in \mathcal{H}^{2\eps}_{n,p}(G)$.
Then, $G \setminus H$ contains a subgraph with at most $\delta n^2 p$ edges satisfying $RE(1/2 + \eps/4)$.
\end{lemma} 

\begin{proof}
Let $1/n\ll1/C\ll\eps,\delta$ and $1/c<1$.
Let $p' \coloneqq \delta p$.
We say that a graph $F$ on $V$ is \textit{good} if it has at most $n^2p' = \delta n^2 p$ edges and, for all $H \in \mathcal{H}^{\eps}_{n,p'}$, the graph $F\setminus H$ satisfies $RE(1/2 + \eps/4)$.
Otherwise, we call it \textit{bad}.
Given any graph $F$ on $V$, let $\hat{F}$ be the graph obtained from $F$ by taking every edge of $F$ independently with probability $\delta$.

Let $\hat{\mathbb{P}}$ be the measure associated with the experiment $\hat{F}$.  
Let $\mathbb{P}_{\text{total}}$ be the product measure obtained from considering the experiments yielding $G_{n,p}$ and $\hat{G}_{n,p}$ (i.e.~with respective measures $\mathbb{P}$ and $\hat{\mathbb{P}}$).
Note that, by definition, the edge distribution of $\hat{G}_{n,p}$ is identical to that of $G_{n,p'}$.
It follows by \cref{lem: edges,lem: re2} that $\mathbb{P}_{\text{total}}[\hat{G}_{n,p} \text{ is good}] = \mathbb{P}[{G}_{n,p'} \text{ is good}] = 1-o(1)$.

Let $\mathcal{F}$ be the collection of all graphs $F$ on $V$ for which $\hat{\mathbb{P}}[\hat{F} \text{ is good}] \ge 3/4$.
Since
\[o(1) = \mathbb{P}_{\text{total}}[\hat{G}_{n,p} \text{ is bad}] \ge \mathbb{P}[G_{n,p} \notin \mathcal{F}]\,\mathbb{P}_{\text{total}}[\hat{G}_{n,p} \text{ is bad} \mid G_{n,p} \notin \mathcal {F}] \ge \mathbb{P}[G_{n,p} \notin \mathcal{F}]/4,\]
we know that $\mathbb{P}[G_{n,p} \notin \mathcal{F}] = o(1)$ or, in other words, $\mathbb{P}[G_{n,p} \in \mathcal{F}] = 1- o(1)$.
Thus, from now on, we consider $G = G_{n,p}$ and condition on the event that $G \in \mathcal{F}$.

Let $H \in  \mathcal{H}^{2\eps}_{n,p}(G)$.
Using \cref{lem: chernoff} and taking a union bound over all vertices in $V$, we have that $\hat{\mathbb{P}}[\hat{G} \cap H \in \mathcal{H}^{\eps}_{n,p'}]=1-o(1)$.
Since $\hat{G}$ is good with probability at least $3/4$, and $\hat{G} \cap H \in \mathcal{H}^{\eps}_{n,p'}$ with probability $1 - o(1)$, there exists a choice of $\hat{G}$ which satisfies these two properties.
For such $\hat{G}$, by the definition of good, the graph $\hat{G} \setminus H$ satisfies $RE(1/2 + \eps/4)$.
Moreover, $\hat{G}$ has at most $\delta n^2 p$ edges and, hence, so does $\hat{G} \setminus H$.
Since $\hat{G} \setminus H \subseteq G \setminus H$, the result follows.
\end{proof}

The proof of \cref{teor:posa} now follows from the previous results.

\begin{proof}[Proof of \cref{teor:posa}]
Let $1/n\ll1/C\ll\delta \ll \eps$.  
Condition on the assertions of \cref{lem: ext,lem: exist} holding with $\eps/2$ instead of $\eps$, which happens a.a.s.
We will show that for any $H \in \mathcal{H}^{\eps}_{n,p}(G)$, the graph $G\setminus H$ is Hamiltonian.

Let $H$ be a graph as above.
By \cref{lem: exist}, there exists a subgraph $G^*$ of $G\setminus H$ which has at most $\delta n^2p$ edges and satisfies property $RE(1/2 + \eps/8)$.
By \cref{lem: ext} we have that $G \setminus H$ complements $G^*$.
Therefore, \cref{prop: comp} implies that $G \setminus H$ is Hamiltonian.
\end{proof}

\section*{Acknowledgements}

We are grateful to Ant\'onio Gir\~ao for some helpful discussions which led us to simplify one of our proofs.

\bibliographystyle{afstylenonumber}
\bibliography{resilience.bib}

\appendix

\section{Proofs of \texorpdfstring{\cref{claim: rotations1,claim: rotations3,lem: ext}}{Lemmas 4.1, 4.2 and 4.8}}

\begin{proof}[Proof of \cref{claim: rotations1}]\COMMENT{This claim is Lemma 23 in \cite{KKO}. I've reproduced it here but we could equally just explain what $\Gamma$ would have to be to apply and quote the theorem there.}
Throughout the proof we write $R^t\coloneqq R_{G,P,v_\ell}^t$ and, for all $x\in V(P)$, $x^+\coloneqq x_{P,v_\ell}^+$ and $x^-\coloneqq x_{P,v_\ell}^-$.
Since $P$ is a longest path, we must have that $N_G(x) \subseteq V(P)$ for all $x\in\mathcal{R}_{G,P,v_\ell}$.
Let $T\coloneqq \{x \in N_{G}(R^t)\setminus R^t \mid x^{-}, x^{+} \notin R^t\}$.
It follows that if $x \in T$, then the segment of $P$ formed by $x^-$, $x$ and $x^+$ is preserved under any sequence of $t$ rotations of $P$.
Since $x \in N_{G}(R^t) \setminus R^t$, it follows that one of $x^{-}, x^{+}$ must be in $R_{G,P,v_\ell}(R^t)=R^{t+1}$.
Now let $T_{+} \coloneqq \{x^{+} \mid x \in T, x^{+} \in R^{t+1}\}$ and $T_{-} \coloneqq \{x^{-} \mid x \in T, x^{-} \in R^{t+1}\}$.
We have that either $|T_{+}| \ge |T|/2$ or $|T_{-}| \ge |T|/2$.
It follows that 
\[|R^{t+1}| \ge \frac{1}{2}|T| \ge \frac{1}{2}|N_{G}(R^t) \setminus R^t| - |R^t|.\qedhere\]
\end{proof}

\begin{proof}[Proof of \cref{claim: rotations3}]
Let $0<1/C\ll\eta \ll 1/c \ll \eps < 1$.
We condition on the event that the following holds for all $X,Y\subseteq V$:
\begin{equation}\label{eq:XYedgesposa}
e'_{G}(X,Y) = |X||Y|p \pm c\sqrt{|X||Y|np}.
\end{equation}
Indeed, \cref{lem: edges} implies this event a.a.s.~occurs.

Let $H\coloneqq G\setminus G'$.
We partition $P$ into $k\coloneqq \eta^{1/2} \log n$ vertex-disjoint intervals $P_1, \dots, P_k$ with $V(P)=\bigcup_{i\in [k]} V(P_i)$, whose lengths are as equal as possible.
By abusing notation, we will also view $P_i$ and $P$ as vertex sets.
Consider any $A\subseteq R_{G',P,v_\ell}^{\eta \log n}$ with $|A|\geq \eps n/100$.
Throughout this proof, we write $R(A)\coloneqq R_{G,P,v_\ell}(A)$.
For each $i\in[k]$, let $\hat{X}_i\subseteq A$ be the collection of all those vertices $x\in A$ for which some edge in $P_i$ is broken in the sequence of rotations resulting in $P_x$.
Let $X_{i,+}$ and $X_{i,-}$ be the collections of all those vertices $x \in A$ such that $P_i$ is unbroken (i.e.~it contains no broken edges) in the sequence of rotations resulting in $P_x$, and where $P_x$ (when directed from $x$ to $v_\ell$) traverses $P_i$ in the original and reverse order, respectively.
Note that $A = \hat{X}_i \cup X_{i,+} \cup X_{i,-}$ for every $i\in[k]$.
Let $I\coloneqq\{i\in[k]:|\hat{X}_i| \ge \eta^{1/4} |A|\}$.

We claim that 
\begin{equation}\label{equa:Isize}
    |I| \le \eta^{3/4} \log{n}.
\end{equation}
Indeed, recall that each vertex in $A$ is obtained by at most $\eta \log n$ rotations of $P$.
By considering the total sum of the number of rotations performed to obtain each different endpoint in $A$ we observe that
\[ \eta^{1/4} |A|\cdot |I| \le |A|\cdot \eta \log{n},\]
which implies \eqref{equa:Isize}.

\begin{claim}\label{eq: eH lowerbound}
We have $e'_H(A,V)\geq|A||V\setminus R(A)|p - \eta^{1/5} n^2 p$.
\end{claim}

\begin{proof}
To prove this, note that, since $P$ is a longest path, we have $e'_{G'}(A,V \setminus P) = 0$. 
Hence, 
\begin{align}\label{eq: edge outside P}
e'_{H}(A,V\setminus P) = e'_{G}(A,V\setminus P).
\end{align}
Throughout this proof, for any $X\subseteq V(P)$ we write $X^+\coloneqq X_{P,v_\ell}^+$ and $X^-\coloneqq X_{P,v_\ell}^-$.
For vertices $v_j \in P_i \cap P_i^{-}$ and $x \in X_{i,+}$, if $x v_{j+1}$ is an edge in $G'$, then we have $v_{j}\in R(A)$.\COMMENT{Note that we intersect with $P^{-}$ here because we do not want to consider the last vertex in the interval. We have no information about the edge between this vertex and the next interval. In the next line we then consider $v_{j+1}$. We intersect with $P^{+}$ here because we do not want to be at the beginning of the interval for $v_{j+1}$ as this would have $v_j$ in the previous interval.}
In other words, $x$ has no edges to $(P_i\cap P_{i}^{+}) \setminus R(A)^{+}$ in the graph $G'$. 
By a similar argument, a vertex $x\in X_{i,-}$ has no edges to $(P_i\cap P_i^{-})\setminus R(A)^{-}$ in $G'$. Thus, we have
\begin{equation}\label{eq: no edges in G'}
e'_{G'}(X_{i,+}, (P_i \cap P_i^{+}) \setminus R(A)^{+}) =0 \quad \text{ and } \quad 
e'_{G'}(X_{i,-}, (P_i \cap P_i^{-}) \setminus R(A)^{-}) =0.
\end{equation}
As $G'=G\setminus H$, this implies that all edges of $G$ between $X_{i,\ast}$ and $(P_i\cap P_i^{\ast}) \setminus R(A)^{\ast}$ belong to $H$, for $\ast\in \{+,-\}$.
As $P_i\cap P_i^{\ast}$ and $P_i^{\ast}$ differ by exactly one vertex, by \eqref{eq: edge outside P} and \eqref{eq: no edges in G'} we have
\begin{align*}
    e'_H(A,V)& \geq e'_G(A, V\setminus P) + \sum_{\ast\in \{+,-\}} \sum_{i=1}^{k} \big(e'_G\big(X_{i,\ast}, \big(P_i \cap P_i^{\ast}\big) \setminus R(A)^{\ast}\big)\\
    &\geq e'_G(A, V\setminus P) + \sum_{\ast\in \{+,-\}} \sum_{i=1}^{k} \big(e'_G\big(X_{i,\ast}, \big(P_i \setminus R(A)\big)^{\ast}\big) - 4kn\\
    &\stackrel{\mathclap{\eqref{eq:XYedgesposa}}}{\geq}|A||V\setminus P|p - c\sqrt{|A|n^2 p} 
    + \!\!\sum_{\ast\in \{+,-\}} \sum_{i=1}^{k} \left(|X_{i,\ast}| |P_i \setminus R(A)| p - c\sqrt{ |X_{i,\ast}||P_i| n p } \right) - 4kn \\
    & \geq |A||V\setminus P|p  + \sum_{\ast\in \{+,-\}} \sum_{i=1}^{k} |X_{i,\ast} | |P_i \setminus R(A)| p -  4c\sqrt{ kn^3p}\\
    & \geq |A||V\setminus P|p + \sum_{i=1}^{k}|A\setminus \hat{X}_i||P_i\setminus R(A)| p - 4c\sqrt{kn^3p},
\end{align*}
where we used that $|P_i| \leq |P|/k+1$ in the penultimate inequality, and the fact that $A = \hat{X}_i\cup X_{i,+}\cup X_{i,-}$ in the final inequality\COMMENT{The second inequality holds since $(P_i \cap P_i^{\ast}) \setminus R(A)^{\ast}$ and $(P_i \setminus R(A))^{\ast}$ differ by at most one element.
For each such vertex, the number of edges between $X_{i,\ast}$ and said vertex is at most $|X_{i,\ast}|\leq n$.
The bound follows by adding over all values of $i$ and $\ast$, and taking into account that each edge may be counted twice.}.
By the definition of $I$, we have $|A\setminus \hat{X}_i| \ge  (1-  \eta^{1/4})|A|$ for all $i\in [k]\setminus I$.
Therefore, we have 
\begin{align*}
e'_H(A,V)&\geq |A||V\setminus P|p + (1-\eta^{1/4}) |A|p\sum_{i\in [k]\setminus I} |P_i\setminus R(A)|   - 4c\sqrt{kn^3 p}\\
& \stackrel{\mathclap{\eqref{equa:Isize}}}{\geq} |A||V\setminus R(A)| p - 2\eta^{1/4} |A|n p - 4c\sqrt{kn^3 p}\\
& \geq |A||V\setminus R(A)| p - \eta^{1/5} n^2 p.
\end{align*}
\COMMENT{We have the $2\eta^{1/4} |A|n p$ as we need to subtract $(1-\eta^{1/4}) |A|p\sum_{i\in  I} |P_i\setminus R(A)| \leq(1-\eta^{1/4}) |A|p\eta^{3/4}\log n(n/k+1)< \eta^{1/4} |A|n p.$}
We obtain the final inequality as $p\geq \log{n}/n$ implies $\sqrt{kn^3 p} \leq \eta^{1/4} n^2 p$.
This proves the claim.
\end{proof}

On the other hand, we have $e'_{G'}(A,V) \geq |A| \delta_{G'}(A)$ and, by \eqref{eq:XYedgesposa}, we have
\[e'_{G}(A,V) \leq |A|np + c \sqrt{n^3p} \leq (1+\eta)|A|np.\]
Therefore, 
\[e'_{H}(A,V) = e'_{G}(A,V) - e'_{G'}(A,V) 
\leq (1+\eta)|A|np -|A|\delta_{G'}(A).\]
Combining this with \cref{eq: eH lowerbound} gives the desired inequality,
\[|R(A)| \geq \delta_{G'}(A) p^{-1} - \eta n - \frac{\eta^{1/5}n^2}{|A|} 
\geq \delta_{G'}(A) p^{-1} - \eps n/10.\qedhere\]
\end{proof}

\begin{proof}[Proof of \cref{lem: ext}]
Let $1/C\ll\delta\ll\eps$.
Let $\mathcal{G}$ be the family of all subgraphs of the form $G \setminus H$, for all $H \in \mathcal{H}^\eps_{n,p}$.
(Note that we have $H \in \mathcal{H}^\eps_{n,p}$ here instead of $H \in \mathcal{H}^\eps_{n,p}(G)$, because this is more convenient for the argument below. But this results in the same family $\mathcal{G}$.)

The probability that the assertion of the lemma fails is 
\begin{align}\label{equa:lastlemmaprob1}
p^* & \coloneqq \mathbb{P}\big[ \bigcup_{R \in RE(1/2 + \eps/4),\ e(R) \leq \delta n^2 p} \big(\{R \subseteq G\} \cap \{\text{some } G' \in \mathcal{G} \text{ does not complement }R\}\big)\big]\nonumber\\
& \le \sum_{R \in RE(1/2 + \eps/4),\ e(R) \leq \delta n^2 p} \mathbb{P} [ \text{some } G' \in \mathcal{G} \text{ does not complement } R \mid R \subseteq G]\,\mathbb{P}[R \subseteq G],
\end{align}
where the union and sum are taken over all labelled graphs $R$ on $V$ which have property $RE(1/2 + \eps/4)$ and at most $\delta n^2 p$ edges.

Let $R$ be a fixed graph on $V$ with property $RE(1/2 + \eps/4)$ and at most $\delta n^2 p$ edges.
Let $P$ be a fixed path on $V$.
If in $R\cup P$ there is a path longer than $P$, then condition \ref{def46condition1} of \cref{def: comp} is already satisfied, so we can assume that there is no such path in $R \cup P$.
Then, by the definition of property $RE(1/2 + \eps/4)$, we can find a set $S_P\subseteq V$ and, for every $v \in S_P$, a corresponding set $T_{v,P}\subseteq V$, as in \cref{REdelta}. 
We can assume that for each $v \in S_P$ we have that $e_R(v, T_{v,P}) = 0$, as otherwise $R$ complements itself and there is nothing more to prove. 
For each $S \subseteq S_P$ with  $|S| = \eps n/8$ let $\mathcal{H}_S \subseteq \mathcal{H}^\eps_{n,p}$ be the collection of graphs $H\in\mathcal{H}^\eps_{n,p}$ for which every $v \in S$ is of the form $v_i$ for $i \ge n/2$ with respect to the ordering of $V(H)$ given in \cref{posa}.
Note that $\bigcup_{S \subseteq S_P:|S|=\eps n/8}\mathcal{H}_S = \mathcal{H}^\eps_{n,p}$.
Thus, given any such $S \subseteq S_P$ and $H \in \mathcal{H}_S$, we have $d_H(v) \le (1/2 - \eps)np$ for all $v \in S$. 
For each such $S \subseteq S_P$ and all $v \in S$, define $T_{v,P,S}\coloneqq T_{v,P} \setminus S$.
Note that $|T_{v,P,S}| \ge (1/2 + \eps /8)n$.

Fix $S \subseteq S_P$ and $v \in S$.
Since $|T_{v,P,S}| \ge (1/2 + \eps /8)n$, by \cref{lem: chernoff} we have
\[\mathbb{P}[e_G(v,T_{v,P,S})\leq np/2\mid R\subseteq G]\leq e^{-\Omega_\eps(np)}.\]
Since $S$ is disjoint from all sets of the form $T_{v,P,S}$, these events are independent for different vertices.
Thus, using that $|S| = \eps n/8$, we can see that
\begin{equation}\label{equa:lastlemmaprob}
    \mathbb{P}[e_G(v,T_{v,P,S})\leq np/2\ \  \text{for all} \ v\in S\mid R\subseteq G]\leq e^{-\Omega_\eps(n^2p)}.
\end{equation}
Note that if there exists $v \in S$ such that $e_G(v,T_{v,P,S})> np/2$, then for each $H\in\mathcal{H}_S$ we have $e_{G\setminus H}(v,T_{v,P,S})>np/2-(1/2-\eps)np>0$.
Therefore, if some $G'\in\mathcal{G}$ does not complement $R$, there must exist some path $P$ on $V$ and some $S\subseteq S_P$ with $|S|=\eps n/8$ such that all of the vertices of $S$ have fewer than $np/2$ neighbours in $T_{v,P,S}=T_{v,P}\setminus S$.
Note that there are at most $n\cdot n!$ choices for the path $P$ and $2^n$ choices for the set $S$.
Taking the union bound over all choices of the path $P$ and the set $S$, by \eqref{equa:lastlemmaprob} we have
\[\mathbb{P}[\text{some } G'\in\mathcal{G} \text{ does not complement } R\mid R \subseteq G] \le n 2^{n}\,n!\, e^{-\Omega_\eps(n^2p)}= e^{-\Omega_\eps(n^2p)}.\]
Combining this with \eqref{equa:lastlemmaprob1}, we have 
\begin{align*}
           p^* &\le e^{-\Omega_\eps(n^2p)} \sum_{R \in RE(1/2 + \eps/4),\ e(R) \leq \delta n^2 p}  \mathbb{P}(R \subseteq G)\\ 
           &\le e^{-\Omega_\eps(n^2p)} \sum_{k=1}^{\delta n^2p} \binom{\binom{n}{2}}{k}p^k 
           \le e^{-\Omega_\eps(n^2p)} \sum_{k=1}^{\delta n^2p} \left(\frac{en^2p}{k}\right)^k\\
           &\le e^{-\Omega_\eps(n^2p)}(\delta n^2 p)\left(\frac{e}{\delta}\right)^{\delta n^2 p}
           \le e^{-\Omega_\eps(n^2p)}e^{O(\delta n^2p\log({1/\delta}))}\\
           &= o(1),
\end{align*}
where the penultimate inequality holds since $(en^2p/k)^k$ is monotone increasing in the range $1\le k\le \delta n^2p$.\COMMENT{
By taking a derivative we see that the maximum of $(\frac{C}{x})^x$ occurs at $x=C/e$.
Indeed, consider the derivative of $\exp{\log ((C/x)^x)}$:
\begin{align*}
    \frac{\diff}{\diff x}\left(e^{x\log C-x\log x}\right)&=e^{x\log C-x\log x}\frac{\diff}{\diff x}(x\log C-x\log x)=e^{x\log (C/x)}(\log C-\log x-1)\\
    &=\left(\frac{C}{x}\right)^x\left(\log\left(\frac{C}{x}\right)-1\right).
\end{align*}
This can only be equal to $0$ when $\log(C/x)=1$, that is, when $x=C/e$.
Furthermore, it is clear that the derivative is positive when $x<C/e$.
In our case, this shows that the function $(\frac{en^2p}{k})^k$ is monotone increasing in $k$ for $k<n^2p$.}
\end{proof}

\bigskip

{\footnotesize \obeylines \parindent=0pt
	\begin{center}
	\begin{tabular}{lll}
        Padraig Condon, Alberto Espuny D\'{i}az,  Daniela K\"uhn and Deryk Osthus	&\ & 	Jaehoon Kim       \\
		School of Mathematics &\ & Department of Mathematical Sciences		  		 	 \\
		University of Birmingham &\ & KAIST 	  			 	 \\
		Birmingham &\ & Daejeon                            			 \\
        B15 2TT      &\ & 		34141				  				\\
		UK &\ & Republic of Korea						      
	\end{tabular}
    \end{center}

\begin{flushleft}
	{\it{E-mail addresses}: \tt{\{pxc644, axe673, d.kuhn, d.osthus\}@bham.ac.uk}, jaehoon.kim@kaist.ac.kr}.
\end{flushleft}
}

\end{document}